\documentclass[11pt,twoside]{amsart}

\usepackage{amsmath}
\usepackage{amsthm}
\usepackage{amsfonts, amssymb}
\usepackage{mathrsfs}
\usepackage[all]{xy}
\usepackage{url}

\usepackage{latexsym}
\usepackage[dvips]{graphicx}

\theoremstyle{plain}
\newtheorem{lema}{Lemma}[section]
\newtheorem{prop}[lema]{Proposition}
\newtheorem{teo}[lema]{Theorem}

\newtheorem{coro}[lema]{Corollary}
\theoremstyle{remark}

\theoremstyle{definition}

\newcommand{\kp}{\mathcal{K}}
\newcommand{\x}{\mathcal{X}}

\begin{document}

\title[On Quillen's Theorem A for posets]{On Quillen's Theorem A for posets}

\author[J.A. Barmak]{Jonathan Ariel Barmak $^{\dagger}$}

\thanks{$^{\dagger}$ Supported by grant KAW 2005.0098 from the Knut 
and Alice Wallenberg Foundation.}

\address{Mathematics Department\\
Kungliga Tekniska h\"ogskolan\\
 Stockholm, Sweden}

\email{jbarmak@kth.se}

\begin{abstract}
A theorem of McCord of 1966 and Quillen's Theorem A of 1973 provide sufficient conditions for a map between two posets to be a homotopy equivalence at the level of complexes. We give an alternative elementary proof of this result and we deduce also a stronger statement: under the hypotheses of the theorem, the map is not only a homotopy equivalence but a simple homotopy equivalence. This leads then to stronger formulations of the simplicial version of Quillen's Theorem A, the Nerve lemma and other known results. 
\end{abstract}

\subjclass[2000]{55U10, 06A07, 55P10, 57Q10, 18B35.}

\keywords{Fiber lemma, posets, simplicial complexes, homotopy equivalences, simple homotopy equivalences}

\maketitle

\section{Introduction}

In his seminal paper \cite{Mcc} M.C. McCord gives a condition for a map between two topological spaces to be a weak homotopy equivalence (a map which induces isomorphisms in all the homotopy groups). Roughly speaking, his theorem (\cite[Theorem 6]{Mcc}) says that if a map is locally a weak homotopy equivalence, then so it is globally. This result allows him to establish the relationship between the homotopy theory of finite topological spaces and finite complexes.

Given a finite poset $X$, the \textit{associated complex} (also called \textit{order complex}) $\kp (X)$ is the simplicial complex whose simplices are the non-empty chains of $X$. An order preserving map $f: X\to Y$ between finite posets induces a simplicial map $\kp (f): \kp (X)\to \kp (Y)$ which coincides with $f$ on vertices. A finite poset $X$ can be considered as a finite topological space and it can be proved from McCord's Theorem that there is a weak homotopy equivalence $\kp (X) \to X$.

The celebrated Theorem A of Quillen \cite{Qui2} establishes a condition for which a functor between two categories induces a homotopy equivalence between the classifying spaces.

Although these powerful and general results apply in very different contexts, they have a particular common application, which is without discussion one of the most useful known tools to study the relation between posets and homotopy theory. The McCord-Quillen Theorem \ref{ppal}, many times referred to as ``Quillen's Fiber lemma", is on one hand McCord's Theorem applied to finite spaces and the covers given by the minimal bases, and on the other hand Quillen's Theorem A applied to finite posets.   

\begin{teo}[McCord '66, Quillen '73] \label{ppal}
Let $f:X\to Y$ be an order preserving map between two finite posets. Suppose that for every $y\in Y$, the complex $\kp(f^{-1}(U_y))$ is contractible. Then $\kp(f)$ is a homotopy equivalence.  
\end{teo}    
Here, $U_y\subseteq Y$ denotes the subset of elements which are smaller than or equal to $y$. Quillen's statement is implicit in \cite{Qui2} and explicit in \cite[Proposition 1.6]{Qui}. Theorem \ref{ppal} has shown to be indispensable in the study of the topology of order complexes of posets. Some important consequences are for example the simplicial version of Theorem A, the so called Nerve lema and Dowker's Theorem on complexes associated to a relation.


Both McCord's Theorem and Quillen's Theorem A have technical nontrivial proofs. In \cite{Wal} Walker gives an elementary proof of Theorem \ref{ppal} using a homotopy version of the Acyclic carrier theorem. In this article we give a different proof of Theorem \ref{ppal}. Our proof is also very basic but the most important consequence is that it can be easily improved to obtain a stronger statement of the theorem.

Whitehead's simple homotopy theory aimed to give a combinatorial description of homotopy types of simplicial complexes. The concepts of simplicial collapse and expansion give rise to the notions of simple homotopy types and simple homotopy equivalences. Simple homotopy equivalent complexes are homotopy equivalent and simple homotopy equivalences are homotopy equivalences, but these implications are strict. CW-complexes were created by Whitehead while he was studying the difference between those concepts. This theory is also of great importance by its applications to combinatorial group theory, differential topology and piecewise-linear topology.

We prove that under the same hypotheses as Theorem \ref{ppal}, the simplicial map $\kp (f)$ is not only a homotopy equivalence but a simple homotopy equivalence.
  
\begin{teo} \label{ppalprima}
Let $f:X\to Y$ be an order preserving map between two finite posets. Suppose that for every $y\in Y$, the complex $\kp(f^{-1}(U_y))$ is contractible. Then $\kp(f)$ is a simple homotopy equivalence.  
\end{teo}

Theorem \ref{ppalprima} originally appears in the author's Thesis \cite[Proposition 6.2.9]{Bar3} formulated in the setting of finite spaces. In Section \ref{seccionppal} we present a self contained proof of the simplicial statement which is more transparent than the one of \cite{Bar3}.   
From this result we immediately obtain stronger formulations of the simplicial version, the Nerve lemma and Dowker's Theorem. 

The key point of our approach is the so called \textit{non-Hausdorff mapping cylinder} of a map between posets introduced by Barmak and Minian in \cite{BM} where it is used to establish the relationship between finite topological spaces (finite posets) and simple homotopy theory of polyhedra.

In the last section we use our ideas to give short proofs of one extension of Theorem \ref{ppal}, studying the case in which the complexes $\kp (f^{-1}(U_y))$ are $n$-connected, and its homological version.

\section{Preliminaries}

The star $st_K(v)$ of a vertex $v$ in a simplicial complex $K$ is the subcomplex of simplices $\sigma \in K$ such that $\sigma \cup \{v\} \in K$. The link $lk_K(v)$ is the subcomplex of $st_K(v)$ of simplices which do not contain $v$. The join of two (disjoint) simplicial complexes $K$ and $L$ is the simplicial complex $K*L$ whose simplices are those of $K$, those of $L$ and unions of a simplex of $K$ with a simplex of $L$. If two complexes are homotopy equivalent, their joins with a third complex are also homotopy equivalent. In particular, the join of a contractible complex with another complex is contractible. For simplicity we will identify a simplicial complex with its geometric realization.

The next basic result follows from the Gluing theorem (see \cite{Bro}).

\begin{prop} \label{push}
Suppose that $K_1$ and $K_2$ are two subcomplexes of a complex $K$, and $K=K_1\cup K_2$. If the inclusion $K_1\cap K_2 \hookrightarrow K_1$ is a homotopy equivalence, then so is $K_2\hookrightarrow K$.  
\end{prop}

The following result is a particular case of the well known fact that natural transformations induce homotopies in the classifying spaces. We include a simple proof for completeness which appears in the author's Thesis \cite[Proposition 2.1.2]{Bar3} and in \cite{BM2}. 

\begin{prop} \label{lema1}
Let $f,g:X\to Y$ be two order preserving maps between finite posets. Suppose that $f(x)\le g(x)$ for every $x\in X$. Then $\kp (f)$ and $\kp (g)$ are homotopic.
\end{prop}
\begin{proof}
Suppose that $f\neq g$. Let $x\in X$ be a maximal point with the property that $f(x)\neq g(x)$. Let $y\in Y$ be an element covering $f(x)$ and such that $y\le g(x)$. Consider the map $h:X\to Y$ which coincides with $f$ in every point different from $x$ and such that $h(x)=y$. By the maximality of $x$, $h$ is order preserving. The simplicial maps $\kp (f)$ and $\kp (h)$ are contiguous (i.e. $\kp (f)(\sigma)\cup \kp (h)(\sigma)\in \kp (Y)$ for every simplex $\sigma \in \kp (X)$) and in particular the linear homotopy between them is well defined and continuous. By induction $\kp (h)\simeq \kp (g)$ and therefore $\kp (f)\simeq \kp (g)$.  
\end{proof}

Given a finite poset $X$, we will denote $U^X_x=\{x' \in X \ | \ x'\le x \}$, $F^X_x=\{x' \in X \ | \ x'\ge x \}$, $\hat{U}^X_x=\{x' \in X \ | \ x'< x \}$ and $\hat{F}^X_x=\{x' \in X \ | \ x'> x \}$. When there is no risk of confusion we will just write $U_x, F_x, \hat{U}_x$ and $\hat{F}_x$. 

\section{An alternative proof of McCord-Quillen Theorem \ref{ppal}}
The idea of our approach is to prove the theorem in some very particular cases in which the map is just an inclusion of a poset into another poset with only one more point. The general case will follow taking compositions of these basic maps and homotopy inverses.   

The next result follows immediately from Theorem \ref{ppal} as it is observed in \cite[Proposition 6.1]{Wal} (see also \cite{BM}) but here we use a different idea (cf. \cite[Proposition 3.10]{BM4}) since we will need it in the proof of the theorem.

\begin{lema} \label{lema2}
Let $X$ be a finite poset and let $x\in X$ be such that $\kp(\hat{U}_x)$ or $\kp(\hat{F}_x)$ is contractible. Then $\kp(X\smallsetminus \{x\}) \hookrightarrow \kp(X)$ is a homotopy equivalence. 
\end{lema}
\begin{proof}
By hypothesis, $lk_{\kp (X)}(x)=\kp(\hat{U}_x)* \kp(\hat{F}_x)$ is contractible. Therefore the inclusion $lk_{\kp (X)}(x)=st_{\kp (X)}(x)\cap \kp (X\smallsetminus \{x\}) \hookrightarrow st_{\kp (X)}(x)$ is a homotopy equivalence. The lemma follows then from Proposition \ref{push}.
\end{proof}

\begin{proof}[\textbf{Proof of Theorem \ref{ppal}}]
Assume that $X$ and $Y$ are disjoint. Consider the \textit{non-Hausdorff mapping cylinder} $B(f)$. The underlying set of the poset $B(f)$ is the union $X\cup Y$. The given ordering within $X$ and $Y$ is kept and for $x\in X$, $y\in Y$ one has $x\le y$ in $B(f)$ if $f(x)\le y$ in $Y$. Let $i:X\hookrightarrow B(f)$ and $j:Y\hookrightarrow B(f)$ be the canonical inclusions.

Let $x_1, x_2, \ldots, x_n$ be a linear extension of $X$ (i.e. an ordering of the elements of $X$ such that $x_r\le x_s$ implies $r\le s$) and denote $Y_r=Y\cup \{x_1, x_2, \ldots, x_r\} \subseteq B(f)$ for each $0\le r\le n$. Then $$\hat{F}_{x_r}^{Y_r}=\{y \ | \ y \ge f(x_r)\}=F_{f(x_r)}^Y.$$
Therefore $\kp (\hat{F}_{x_r}^{Y_r})=\kp (F_{f(x_r)}^Y)$ is a cone and in particular, contractible. By Lemma \ref{lema2}, $\kp (Y_{r-1}) \hookrightarrow \kp (Y_{r})$ is a homotopy equivalence and then the inclusion $\kp(j): \kp (Y)=\kp (Y_0) \hookrightarrow \kp (Y_n)=\kp (B(f))$ is also a homotopy equivalence.

Now let $y_1,y_2, \ldots, y_m$ be a linear extension of $Y$ and let $X_r=X \cup \{y_{r+1},y_{r+2}, \ldots, y_m\} \subseteq B(f)$ for every $0\le r\le m$. Then $$\hat{U}_{y_{r}}^{X_{r-1}}=\{x \ | \ f(x)\le y_{r}\}=f^{-1}(U_{y_r}^Y).$$ By hypothesis, $\kp (\hat{U}_{y_{r}}^{X_{r-1}})$ is contractible. By Lemma \ref{lema2}, $\kp (X_r) \hookrightarrow \kp (X_{r-1})$ is a homotopy equivalence and then so is $\kp (i): \kp (X)=\kp (X_m) \hookrightarrow \kp (X_0)=\kp (B(f))$.

Since $i(x)\le jf(x)$ for every $x\in X$, by Proposition \ref{lema1}, $\kp(i)\simeq \kp(jf)=\kp (j) \kp (f)$. Hence, $\kp (f)$ is a homotopy equivalence. 
\end{proof}

\section{A simple stronger statement} \label{seccionppal}

We will show that the proof of Theorem \ref{ppal} can be easily modified to obtain the stronger Theorem \ref{ppalprima}.

If $K$ is a finite simplicial complex with a simplex $\tau$ which is a proper face of a unique simplex $\sigma$, we say that there is an elementary collapse from $K$ to the subcomplex $L\subset K$ which is obtained from $K$ by removing the simplices $\sigma$ and $\tau$. If there is a sequence of elementary collapses from a complex $K$ to a subcomplex $L$, we say that $K$ collapses to $L$. Two complexes have the same simple homotopy type if it is possible to obtain one from the other by performing collapses and their inverses (expansions).  
A class of maps $\mathcal{C}$ between topological spaces is said to satisfy the $2$-out-of-$3$ property if whenever there are three maps $f,g,h$ such that the composition $fg$ is well defined, $fg\simeq h$ and two of the three maps are in $\mathcal{C}$, then so is the third. The class of simple homotopy equivalences is the smallest class satisfying the $2$-out-of-$3$ property and containing all the inclusions $L \hookrightarrow K$ where $K$ is a complex and $L$ is a subcomplex which expands to $K$. For basic properties on simple homotopy theory we encourage the readers to consult \cite{Coh}.

Theorem (20.1) of \cite{Coh} states that if $L$ is a subcomplex of a complex $K$, the inclusion $L\hookrightarrow K$ is a homotopy equivalence and every connected component of the space $K \smallsetminus L$ is simply connected, then $L\hookrightarrow K$ is a simple homotopy equivalence. 

From this result and Lemma \ref{lema2} we obtain a refined statement of Lemma \ref{lema2}. Note that if $X$ is a finite poset and $x\in X$, then the space $\kp (X)\smallsetminus K(X\smallsetminus \{x\})$ is the open star of $x$ in $\kp (X)$ which is contractible. 

\begin{lema} \label{lema2prima}
Let $X$ be a finite poset and let $x\in X$ be such that $\kp(\hat{U}_x)$ or $\kp(\hat{F}_x)$ is contractible. Then $\kp(X\smallsetminus \{x\}) \hookrightarrow \kp(X)$ is a simple homotopy equivalence. 
\end{lema}

\begin{proof}[\textbf{Proof of Theorem \ref{ppalprima}}]
The proof is essentially the same as the one of Theorem \ref{ppal}. We use Lemma \ref{lema2prima} instead of Lemma \ref{lema2} and the $2$ out-of $3$ property. Specifically, by Lemma \ref{lema2prima}, the inclusions $\kp (Y_{r-1})\hookrightarrow \kp (Y_r)$ and $\kp (X_{r})\hookrightarrow \kp (X_{r-1})$ are simple homotopy equivalences. Since this class is closed under compositions, $\kp (i)$ and $\kp (j)$ are also simple equivalences and being $\kp (i)\simeq \kp (j) \kp (f)$, so is $\kp (f)$.
\end{proof}


Given a finite simplicial complex $K$, its \textit{associated poset} $\x (K)$ (also known as face poset) is the poset of simplices of $K$ ordered by containment. A simplicial map $\varphi : K\to L$ also has associated an order preserving map $\x (\varphi): \x (K) \to \x (L)$ defined by $\x (\varphi) (\sigma)=\varphi (\sigma)$. Note that $\kp (\x (K))$ coincides with the barycentric subdivision $K'$. It is a standard fact that a simplicial complex and its barycentric subdivision are simple homotopy equivalent. Moreover, a simplicial map $\varphi :K\to L$ is a simple homotopy equivalence if and only if the induced map $\varphi '=\kp (\x (\varphi)): K'\to L'$ is a simple homotopy equivalence.

We deduce the following result which is a stronger version of the simplicial statement of Quillen's Theorem A \cite{Qui2}. It also sharpens Theorem 4.3.14 of \cite{Bar3} which requires a more restrictive hypothesis on the map $\varphi$.

\begin{teo} \label{simplicialprima}
Let $\varphi : K \to L$ be a simplicial map between two finite complexes. Suppose that the preimage of each closed simplex of $L$ is contractible. Then $\varphi$ is a simple homotopy equivalence.
\end{teo}
\begin{proof}
We show that the associated map $\x (\varphi)$ satisfies the hypotheses of Theorem \ref{ppalprima}. Given $\sigma \in \x (L)$, $$\kp (\x (\varphi)^{-1}(U_\sigma))=\kp (\x(\varphi ^{-1}(\sigma)))$$ is the barycentric subdivision of $\varphi ^{-1}(\sigma)$ which is contractible by hypothesis. Thus, $\varphi '$ is a simple homotopy equivalence and then so is $\varphi$.
\end{proof}

The original result of Quillen concludes under the same hypotheses that $\varphi$ is a homotopy equivalence. 

Another consequence of Theorem \ref{ppalprima} is the following improvement of the Nerve lemma proved by Borsuk (see \cite[Theorem 10.6]{Bjo}). Recall that the \textit{nerve} of a family $\mathcal{U}=\{U_i\}_{i\in I}$ of subsets of a set is the simplicial complex $\mathcal{N}(\mathcal{U})$ whose simplices are the finite subsets $J$ of $I$ such that $\bigcap\limits_{i\in J}U_i\neq \emptyset$. Given a poset $X$, we denote by $X^{op}$ the poset with the reversed order. 

\begin{teo} \label{nerveprima}
Let $K$ be a finite simplicial complex and let $\mathcal{U}=\{L_i\}_{i\in I}$ be a finite family of subcomplexes of $K$ such that $\bigcup\limits_{i\in I} L_i=K$ and such that every intersection of elements of $\mathcal{U}$ is empty or contractible. Then $K$ has the same simple homotopy type as $\mathcal{N}(\mathcal{U})$.
\end{teo}
\begin{proof}
The map $\x (K) \to \x (\mathcal{N}(\mathcal{U})) ^{op}$ that maps a simplex $\sigma \in K$ into $\{i \in I \ | \ \sigma \in L_i\}$ satisfies the hypotheses of Theorem \ref{ppalprima}. Therefore there is a simple homotopy equivalence from $K'$ to $\mathcal{N}(\mathcal{U})'$.
\end{proof}

We deduce then a stronger version of Dowker's Theorem \cite{Dow} (see also \cite[Theorem 10.9]{Bjo}).

\begin{teo}
Let $X$ and $Y$ be two finite sets and let $R\subseteq X\times Y$ be a relation. Consider the simplicial complex $K$ whose simplices are the subsets of $X$ of elements which are related to a same element of $Y$ and symmetrically, the simplices of the complex $L$ are subsets of $Y$ of elements related to a same element of $X$. Then $K$ and $L$ have the same simple homotopy type.  
\end{teo}
\begin{proof}
Consider for every $x\in X$ the full subcomplex $\sigma _x \subseteq L$ spanned by the elements related to $x$ (it is a simplex if non-empty). These subcomplexes cover $L$ and clearly any intersection of them is empty or contractible. The nerve of this cover is $K$ and the result follows then from Theorem \ref{nerveprima}. 
\end{proof}

Rota's Crosscut theorem can also be improved by a direct application of Theorem \ref{nerveprima} (see \cite[Theorem 10.8]{Bjo}). 

In general, the usual applications of Theorem \ref{ppal} have now a more precise statement, in particular Quillen's original results on the poset of nontrivial $p$-subgroups of a group \cite{Qui}. Hopefully Theorem \ref{ppalprima} could help in the search for a complete proof of Quillen's Conjecture \cite[Conjecture 2.9]{Qui}. 

\section{Two more applications}

Other versions of the McCord-Quillen Theorem can be obtained by modifying the hypotheses on the subcomplexes $\kp (f^{-1}(U_y))$. The following result was proved by Bj\"orner \cite[Theorem 2]{Bjo2} using the homotopy version of the Acyclic carrier theorem. We exhibit here an alternative proof using our approach to Theorem \ref{ppal}. Recall that a continuous map $f:X\to Y$ between two topological spaces is said to be an $n$-equivalence if for every $x\in X$, the induced map $\pi _i (X,x)\to \pi _i (Y,f(x))$ is an isomorphism for $i<n$ and an epimorphism for $i=n$. 

\begin{teo}[Bj\"orner] \label{ppalbjorner}
Let $f:X\to Y$ be an order preserving map between two finite posets and let $n$ be a nonnegative integer. Suppose that for every $y\in Y$, the complex $\kp(f^{-1}(U_y))$ is $n$-connected. Then $\kp(f)$ is an $(n+1)$-equivalence.  
\end{teo}

We need a third version of Lemma \ref{lema2}.

\begin{lema} \label{lema2bjorner}
Let $n\ge 0$, let $X$ be a finite poset and let $x\in X$ be such that $\kp(\hat{U}_x)$ is $n$-connected. Then $\kp(X\smallsetminus \{x\}) \hookrightarrow \kp(X)$ is an $(n+1)$-equivalence.
\end{lema}
\begin{proof}
The link $lk_{\kp (X)}(x)=\kp (\hat{U}_x)*\kp(\hat{F}_x)$ is also $n$-connected by \cite[Lemma 2.3]{Mil} and therefore the pair $(st_{\kp (X)}(x), lk_{\kp (X)}(x))$ is $(n+1)$-connected. We can assume that $\kp (X)$ is connected and therefore, $(\kp (X\smallsetminus \{x\}), lk_{\kp (X)}(x))$ is $0$-connected. By the Excision theorem for homotopy groups \cite[Theorem 4.23]{Hat}, the map $\pi_i (st_{\kp (X)}(x), lk_{\kp (X)}(x))\to (\kp (X), \kp (X\smallsetminus \{x\}))$ induced by the inclusion is an isomorphism for $i<n+1$ and an epimorphism for $i=n+1$. Thus, $(\kp (X), \kp (X\smallsetminus \{x\}))$ is $(n+1)$-connected and the lemma follows.
\end{proof}

\begin{proof}[Proof of Theorem \ref{ppalbjorner}]
As in the proof of Theorem \ref{ppal}, $\kp (j)$ is a homotopy equivalence and since composition of $(n+1)$-equivalences is again an $(n+1)$-equivalence, by Lemma \ref{lema2bjorner} $\kp (i)$ is an $(n+1)$-equivalence. Since $\kp (i)\simeq \kp (j) \kp (f)$, $\kp (f)$ is an $(n+1)$-equivalence.
\end{proof}

Before proving the homological analogous of Theorem \ref{ppalbjorner} due to Quillen \cite{Qui}, we state a fourth version of Lemma \ref{lema2}.

\begin{lema}\label{lema2homo}
Let $n\ge 0$, let $X$ be a finite poset and let $x\in X$ be such that the reduced (integral) homology groups $\widetilde{H}_i(\kp(\hat{U}_x))$ are trivial for $i\le n$. Then the map $\widetilde{H}_i(\kp(X\smallsetminus \{x\})) \hookrightarrow \widetilde{H}_i(\kp(X))$ induced by the inclusion is an isomorphism for $i\le n$ and an epimorphism for $i=n+1$.
\end{lema}
\begin{proof}
The groups $\widetilde{H}_i(lk_{\kp (X)}(x))=\widetilde{H}_i(\kp (\hat{U}_x)*\kp(\hat{F}_x))$ are trivial for $i\le n$ by \cite[Lemma 2.1]{Mil}. The result then follows from the Mayer-Vietoris sequence for the decomposition $\kp (X)=\kp (X\smallsetminus {x})\cup st_{\kp (X)}(x)$. 
\end{proof}

Using again the ideas of the proof of Theorem \ref{ppal} we deduce the following result.

\begin{teo}[Quillen]
Let $f:X\to Y$ be an order preserving map between two finite posets and let $n$ be a nonnegative integer. Suppose that for every $y\in Y$, the reduced homology groups $\widetilde{H}_i(\kp(f^{-1}(U_y)))$ are trivial for $i\le n$. Then $\kp(f)_*:\widetilde{H}_i(\kp(X))\to \widetilde{H}_i(\kp(Y))$ is an isomorphism for $i\le n$ and an epimorphism for $i=n+1$. 
\end{teo}

\begin{coro}
Let $f:X\to Y$ be an order preserving map between two finite posets. If $\kp (f^{-1}(U_y))$ is acyclic for every $y\in Y$, then $\kp (f)$ induces isomorphisms in all the homology groups.
\end{coro}

\textbf{Acknowledgments}
\medskip

The basic ideas of this article were originated in my PhD work with my advisor Gabriel Minian. I am grateful to him for his encouragement, advices and constant support, then and now. I also want to thank Anders Bj\"orner for some useful suggestions and comments.


\begin{thebibliography}{99}

\bibitem{Bar3} J.A. Barmak. \textit{Algebraic topology of finite topological spaces and applications}.
		PhD Thesis, Facultad de Ciencias Exactas y Naturales, Universidad de Buenos Aires (2009).

\bibitem{BM4} J.A. Barmak and E.G. Minian. \textit{One-point reductions of finite spaces, h-regular CW-complexes and collapsibility}.
		Algebr. Geom. Topol. 8 (2008), 1763-1780.

\bibitem{BM} J.A. Barmak and E.G. Minian. \textit{Simple homotopy types and finite spaces}.
		Adv. Math. 218 (2008), Issue 1, 87-104.

\bibitem{BM2} J.A. Barmak and E.G. Minian. \textit{Strong homotopy types, nerves and collapses}.
		arXiv:0907.2954v1

\bibitem{Bjo2} A. Bj\"orner. \textit{Nerves, fibers and homotopy groups}.
		J. Combin. Theory, Ser. A, 102 (2003), 88-93.

\bibitem{Bjo} A. Bj\"orner. \textit{Topological methods}.
		Handbook of combinatorics (ed. R. Graham, M. Gr\"otschel and L. Lov\'asz; North-Holland, Amsterdam).

\bibitem{Bro} R. Brown. \textit{Elements of modern topology}.
		McGraw-Hill, London, 1968.

\bibitem{Coh} M.M. Cohen. \textit{A Course in Simple Homotopy Theory}.
    Springer-Verlag New York, Heidelberg, Berlin (1970).

\bibitem{Dow} C.H. Dowker. \textit{Homology groups of relations}.
		Ann. of Math. 56 (1952), 84-95.

\bibitem{Hat} A. Hatcher. \textit{Algebraic Topology}.
		Cambridge University Press (2002).

\bibitem{Mcc} M.C. McCord. \textit{Singular homology groups and homotopy groups of finite topological spaces}.
    Duke Math. J. 33 (1966), 465-474.

\bibitem{Mil} J.W. Milnor \textit{Construction of universal bundles II}.
	 Ann. of Math. 63 (1956), 430-436. 

\bibitem{Qui2} D. Quillen. \textit{Higher algebraic $K$-theory, I: Higher $K$-theories}.
		Lect. Notes in Math. 341 (1972), 85-147.

\bibitem{Qui} D. Quillen. \textit{Homotopy properties of the poset of nontrivial $p$-subgroups of a group}.
		Adv. Math. 28 (1978), 101-128. 

\bibitem{Wal} J.W. Walker. \textit{Homotopy type and Euler characteristic of partially ordered sets}.
		European J. of Combinatorics 2 (1981), 373-384.

\end{thebibliography}
\end{document}